\documentclass[a4paper,11pt]{amsart}

\usepackage{amscd,amsmath,amssymb,fancyhdr}
\usepackage[utf8]{inputenc}
\usepackage{amsfonts}
\usepackage{amsthm}
\usepackage{accents}
\usepackage{graphicx}
\usepackage{float}
\usepackage{indentfirst}
\usepackage{dsfont}
\usepackage{tikz}
\usetikzlibrary{matrix}
\usepackage{enumerate}

\usepackage[top=1.5in, bottom=1in, left=0.9in, right=0.9in]{geometry}

\usepackage{scalerel,stackengine}
\stackMath
\newcommand\reallywidehat[1]{
	\savestack{\tmpbox}{\stretchto{
			\scaleto{
				\scalerel*[\widthof{\ensuremath{#1}}]{\kern-.6pt\bigwedge\kern-.6pt}https://www.overleaf.com/project/6314d9f019e04269bfa2f21e
				{\rule[-\textheight/2]{1ex}{\textheight}}
			}{\textheight}
		}{0.5ex}}
	\stackon[1pt]{#1}{\tmpbox}
}

\usepackage[backref=page]{hyperref}
\renewcommand*{\backref}[1]{}
\renewcommand*{\backrefalt}[4]{
	\ifcase #1 (Not cited.)
	\or        (Cited on page~#2.)
	\else      (Cited on pages~#2.)
	\fi}

\hypersetup{
	colorlinks   = true,
	citecolor    = magenta
}

\numberwithin{equation}{section}

\renewcommand{\to}{\longrightarrow}

\newcounter{Mycounter}[section]

\newcounter{lemma}[section]
\newcounter{claim}[section]
\newcounter{sublemma}[section]
\newcounter{corollary}[section]
\newcounter{theorem}[section]
\newcounter{conjecture}[section]
\newcounter{proposition}[section]
\newcounter{definition}[section]
\newcounter{example}[section]
\newcounter{remark}[section]
\newcounter{problem}[section]
\newcounter{question}[section]
\makeatletter
    \@addtoreset{equation}{section}
    \@addtoreset{footnote}{section}
\makeatother

\usetikzlibrary{arrows,chains,matrix,positioning,scopes}

\makeatletter
\tikzset{join/.code=\tikzset{after node path={
			\ifx\tikzchainprevious\pgfutil@empty\else(\tikzchainprevious)
			edge[every join]#1(\tikzchaincurrent)\fi}}}
\makeatother

\tikzset{>=stealth',every on chain/.append style={join},
	every join/.style={->}}

\begin{document}
	
\title{A note on compatibility of special Hermitian structures}
	
\author{Daniele Angella}
\address[Daniele Angella]{
Dipartimento di Matematica e Informatica ``Ulisse Dini"\\
Universit\` a degli Studi di Firenze\\
viale Morgagni 67/A\\
50134 Firenze, Italy
}
\email{daniele.angella@unifi.it}
\email{daniele.angella@gmail.com}
	
\author{Alexandra Otiman}
\address[Alexandra Otiman]{
Institut for Matematik and Aarhus Institute of Advanced Studies\\
Aarhus University\\
8000, Aarhus C, Denmark\\
and\\
Institute of Mathematics ``Simion Stoilow'' of the Romanian Academy\\
21 Calea Grivitei Street\\
010702, Bucharest, Romania
}
\email{alexandra.otiman@imar.ro}
\email{aiotiman@aias.au.dk}

\thanks{
D. A. is partially supported by project PRIN2017 ``Real and Complex Manifolds: Topology, Geometry and holomorphic dynamics'' (code 2017JZ2SW5) and by GNSAGA of INdAM.
A. O. is supported by the European Union’s Horizon 2020 research and innovation programme under the Marie Skłodowska-Curie grant agreement No 754513 and Aarhus
University Research Foundation and by a grant of the Ministry of Research and Innovation, CNCS-UEFISCDI, project no. P1-1.1-PD-2021-0474}

\dedicatory{Dedicated to Professor Victor Vuletescu for his 60th birthday.}

\begin{abstract}
We prove that a compact Vaisman manifold $(M, J)$ cannot admit some type of special Hermitian metrics, such as special $k$-Gauduchon metrics, $p$-K\" ahler forms, Hermitian-symplectic or strongly Gauduchon metrics compatible to the same complex structure $J$. In particular, it cannot admit pluriclosed or balanced metrics. We also investigate the interplay between locally conformally symplectic forms taming the complex structure $J$ and special Hermitian structures.

\hfill



\end{abstract}

\maketitle

\section{Introduction}

In the past decades, constant effort has been put into finding suitable replacements of K\" ahler metrics on compact complex manifolds of non-K\"ahlerian type, which may possibly be canonical and then help towards classification results. Special Hermitian metrics emerged by weakening the K\" ahler condition with either cohomological or analytic conditions. Among the most intensively studied structures, there are {\it pluriclosed} (also known as SKT), {\it balanced}, {\it locally conformally K\" ahler} (LCK for short) and {\it Hermitian-symplectic}. For further details and motivation on each of these geometries, we refer {\itshape e.g.} to \cite{bis89, st10, m82, garcia-fernandez, ov-book, bel00, li-zhang} and references therein.

The aim of this note is to study the interactions between several of these special structures and answer partially the question when they can co-exist on a given compact manifold, as long as they are compatible with the same complex structure. We focus essentially on two kinds of ambient complex manifolds: {\it Vaisman manifolds}, that represent a special type of locally conformally K\" ahler manifolds; and complex manifolds admitting {\it locally conformally symplectic forms taming the complex structure}, which are a  non-metric analogue of LCK structures.
A prototype result in this context is the classical Vaisman theorem \cite[Theorem 2.1]{vai80}, stating that a compact manifold carrying a locally conformally K\"ahler, non-globally conformally K\"ahler, metric cannot be K\" ahlerian.
The motivation for studying this problem comes from the recent conjecture \cite[Conjecture 5.3]{fgv22} stating that a compact, non-K\"ahler, locally conformally K\" ahler manifold cannot be endowed with various different special K\" ahler metrics, such as balanced, pluriclosed or $p$-pluriclosed. Since the locally conformally K\" ahler class sits inbetween two geometries with distinctive sets of techniques, Vaisman and locally conformally symplectic respectively, 
$$ \text{Vaisman} \quad \subset \quad \text{LCK} \quad \subset \quad \text{LCS taming $J$}$$
we endeavour to study the interplay of these with certain non-K\" ahler structures, in order to shed more light on the LCK setting.
On compact complex nilmanifolds, the interaction between locally conformally K\"ahler metrics and several other special non-K\"ahler Hermitian metrics  was studied in \cite{ornea-otiman-stanciu}.

It would also be interesting to study LCK Calabi-Yau solvmanifolds. General Calabi-Yau Vaisman manifolds were studied in \cite{nico} and Vaisman solvmanifolds with holomorphically trivial canonical bundle were studied in \cite{andrada-origlia}, and by Andrada and Tolcachier.

\bigskip

The paper is organized as follows.
In Section \ref{preli}, we briefly present the geometric structures that we will consider in the following.
In Section \ref{sec:metrics}, we prove the incompatibility between Vaisman metrics and several special non-K\" ahler structures, including  pluriclosed, Hermitian-symplectic, astheno-K\"ahler, balanced, holomorphic symplectic, strongly Gauduchon, see \ref{thm:vaisman-pp}.
In Section \ref{sec:LCS-taming}, we study locally conformally symplectic structures taming the complex structure. We show that strict structures of this kind cannot exist on compact complex non-K\"ahler manifolds satisfying a certain cohomological condition, see \ref{generalized-Vaisman}, and on certain fibre bundles, see \ref{prop:LCS-taming-fibrations}. Moreover, we show that the existence of LCS forms taming the complex structure is stable by blowing up special submanifolds, see \ref{prop:LCS-taming-blowup} and by small analytic deformations, see \ref{prop:LCS-taming-deformations}. Finally, we prove that Kato manifolds admit LCS forms that tame the complex structure if and only if they admit LCK metrics, see \ref{prop:kato}.

\bigskip

\noindent {\it Acknowledgements.} The authors are very grateful to Victor Vuletescu for the many discussions over the years, which have been a great source of inspiration, both in mathematics and in life: this note is dedicated to celebrate Victor's 60th birthday. {\it La mul\c{t}i ani! Buon compleanno!}\\
Many thanks also to Liviu Ornea and Miron Stanciu for their comments and suggestions.
We are very grateful to Adrián Andrada and Alejandro Tolcachier for pointing out a mistake in a previous version of this paper.

\section{Preliminaries and notation}\label{preli}

We outline below several types of special non-K\" ahler geometries as well as some non-metric counterparts that have been studied as interesting generalizations of K\" ahler metrics.
As a matter of notation, we are fixing a differentiable manifold $M$ of dimension $2n$ endowed with a structure of complex manifold, equivalently with an integrable almost complex structure $J$. We confuse a Hermitian metric $g$ with its associated $(1,1)$-form $\omega=g(J\cdot, \cdot)$.

\begin{description}

 \item[Locally conformally Kähler] A Hermitian metric $g$ is called {\it locally conformally K\"ahler} (LCK for short) if, locally, it is conformally equivalent to a K\"ahler metric, see \cite{vai76}. Analytically, this is equivalent to ask that  there exists a closed one-form $\theta$ (called the {\em Lee form}) such that $\omega$ satisfies $d\omega=\theta \wedge \omega$. If $\theta$ is exact, the metric is called {\em globally conformally K\"ahler} (gcK for short), otherwise we shall call it {\it strictly} LCK. If the Lee form is parallel w.r.t. the Levi-Civita connection of $g$, the LCK metric is called {\em Vaisman}. For a up-to-date reference to LCK and Vaisman geometries, we refer to \cite{ov-book}.

 \item[Pluriclosed] also referred to in the literature as {\em strongly K\"ahler with torsion}.  A Hermitian metric $g$ is called {\it pluriclosed} if its fundamental form $\omega$ satisfies $dd^c \omega=0$. See {\em e.g.}  \cite{bis89, garcia-fernandez-streets} for their role in complex geometry and in generalized geometry.
 
 \item[$k$-Gauduchon] fix $k\in \{1,\ldots,n-1\}$, a Hermitian metric $g$ is called {\it k-Gauduchon} if its fundamental form $\omega$ satisfies $dd^c\omega^{k} \wedge \omega^{n-k-1}=0$. They were first introduced in \cite{fww13} and they generalize the notion of {\it Gauduchon metric}, which corresponds precisely to $k=n-1$. The latter exists and is unique in any conformal class on a compact manifold, up to multiplication by positive constants, due to the celebrated result of Gauduchon \cite{gau77}. Among the special cases of $k$-Gauduchon metrics. we have pluriclosed (correponding to $1$-Gauduchon) and {\it astheno-K\" ahler} (defined by $dd^c\omega^{n-2}$=0, hence in particular $(n-2)$-Gauduchon).

 \item[Balanced] also known as {\em semi-K\"ahler}. A Hermitian metric $g$ is called {\it balanced} if its fundamental form $\omega$ satisfies $d\omega^{n-1}=0$, or equivalently if $\omega$ is co-closed with respect to the metric itself. See  {\em e.g.} M.L. Michelson's paper \cite{m82}.
 The balanced condition can be generalized to the notion of {\it $p$-K\"ahler} structure, by requiring the existence of a closed $(p, p)$-form `transverse' to the complex structure, see \cite{aa} for more details. Here, we only notice here that a complex manifold is $1$-K\"ahler if and only if it is K\"ahler, while it is $(n-1)$-K\"ahler if and only if it is balanced, see \cite[Proposition 1.15]{aa}.

 \item[Strongly Gauduchon] A Hermitian metric $g$ is called {\it strongly Gauduchon} if its fundamental form $\omega$ satisfies $\partial \omega^{n-1}$ is $\overline{\partial}$-exact, see \cite{pop13}.

 \item[Hermitian-symplectic] A \textit{Hermitian-symplectic structure} is a real two-form $\Omega$ such that $d\Omega = 0$ and its $(1,1)$-component $\Omega^{(1, 1)}$ is positive. This latter property is also known as {\it $\Omega$ taming $J$}, see \cite{st10} and is equivalent to $\Omega(X, JX)>0$, for any real non-zero vector field $X$. In particular, $\Omega$ is a symplectic form and $\Omega^{(1, 1)}$ yields a pluriclosed metric.
 It is known that a compact complex surface is Hermitian-symplectic if and only if it is K\"ahler; it is an open question, see \cite[Question 1.7]{st10} and \cite[page 678]{li-zhang}, whether there exist compact complex non-K\"ahler manifolds of complex dimension greater than $2$ that admits Hermitian-symplectic structure.

 \item[Locally conformally symplectic (LCS) form taming the complex structure] A real $2$-form $\Omega$ satisfying $d\Omega=\theta \wedge \Omega$ for a closed real one-form $\theta$ is \textit{taming the complex structure} if $\Omega^{(1, 1)}$ is positive, or equivalently if $\Omega(X, JX)>0$, for any real non-zero vector field $X$.

\end{description}

Notice that the intersection of the above notions contains the notion of {\it K\"ahler} metric, namely a Hermitian metric $g$ such that its fundamental form $\omega$ satisfies $d\omega=0$.

We will also consider {\em holomorphic symplectic} structures, namely symplectic forms that are holomorphic $2$-forms with respect to the complex structure.

\medskip

\noindent{\bf Conventions.} Throughout the paper we shall use the conventions from \cite[(2.1)]{bes87} for the complex structure $J$ acting on complex forms on a complex manifold $(M, J)$. Namely:
\begin{itemize}
\item $J\alpha=\sqrt{-1}^{q-p} \alpha$, for any $\alpha \in \Lambda^{p, q}_{\mathbb{C}}M$, or equivalently $$J\eta(X_1, \ldots, X_p) = (-1)^p\eta(JX_1, \ldots, JX_p);$$
\item the fundamental form of a Hermitian metric is given by $\omega (X, Y): = g(JX, Y)$;
\item the operator $d^c$ is defined as $d^c := -J^{-1}dJ$, where $J^{-1} = (-1)^{\deg \alpha} J$. 
\end{itemize}

\section{Vaisman manifolds and special Hermitian metrics}\label{sec:metrics}

Let $(M,J)$ be a compact complex manifold endowed with a Vaisman metric $g$ with associate $(1,1)$-form $\omega_0$ and Lee form $\theta$.
Vaisman geometry is strongly related to K\" ahler geometry via the canonical foliation $\mathcal{F}$ defined by $\theta^{\sharp}$, the metric dual of the Lee form, and $J\theta^{\sharp}$. Assuming that $\|\theta\|=1$, which we can always achieve by possibly rescaling the metric since $\|\theta\|$ is constant \cite{vai82}, let us recall that the Vaisman metric $\omega_0$ satisfies (see {\it e.g.} \cite[Proposition 8.3]{ov-book})
\begin{equation}\label{Vaismaneq}
\omega_0=\theta \wedge J\theta - dJ\theta.
\end{equation}
The $(1, 1)$-form $\eta:=-dJ\theta$ is an exact semi-positive form, with kernel precisely $T\mathcal{F}$ and defining a transverse K\" ahler metric with respect to $\mathcal{F}$, see \cite[Theorem 3.1]{vai82}. This $(1, 1)$-form turns out to be the source of incompatibility between Vaisman metrics and several special non-K\" ahler structures.

\begin{theorem}\label{thm:vaisman-pp}
Let $M$ be a compact Vaisman manifold of complex dimension $n$. Then:
\begin{enumerate}[(i)]
    \item $M$ cannot admit a positive $(1, 1)$-form $\omega$ satisfying $dd^c \omega^k = 0$ for some $1 \le k \le n-2$. In particular, when $n\geq3$, it cannot admit pluriclosed, Hermitian-symplectic or astheno-K\"ahler metrics.
    \item $M$ cannot admit positive closed $(p, p)$-forms, for any $1 \leq p \leq n-1$. In particular, it cannot admit balanced metrics; when $n \geq 4$, it cannot admit holomorphic symplectic structures.
    \item $M$ cannot admit strongly Gauduchon metrics. 
\end{enumerate}
\end{theorem}

\begin{proof}
Let $\omega_0$ be a Vaisman metric such that its corresponding Lee form $\theta$ satisfies $||\theta||_{\omega_0}=1$. Then using \eqref{Vaismaneq} and the fact that $dJ\theta$ is a $(1, 1)$-form, hence $J$-invariant, we have the following equalities for any $1 \le k \le n-2$:
\begin{align}\label{thm:vaisman-pp-eq1}
dd^c \omega_0^{n-k-1} & = dJd\omega_0^{n-k-1} =(n-k-1) dJ(\theta \wedge \omega_0^{n-k-1}) \\
& = (-1)^{n-k-1}(n-k-1) d(J\theta \wedge (dJ\theta)^{n-k-1}) \nonumber\\
& = -(n-k-1) (-dJ\theta)^{n-k} \leq 0 . \nonumber
\end{align}

\begin{enumerate}[(i)]
\item Assume there exists a positive $(1, 1)$-form $\omega$ with $dd^c \omega^k = 0$, for some $1 \le k \le n-2$. Then, since $\omega^k$ is a positive $(k, k)$-form,
\[
\int_M dd^c \omega_0^{n-k-1} \wedge \omega^k \stackrel{\eqref{thm:vaisman-pp-eq1}}{=} -(n-k-1) \int_M (-dJ\theta)^{n-k} \wedge \omega^k < 0.
\]
However, by Stokes' theorem,
\[
\int_M dd^c \omega_0^{n-k-1} \wedge \omega^k = \int_M \omega_0^{n-k-1} \wedge dd^c \omega^k = 0,
\]
a contradiction.
 \item Assume there exists a positive closed $(p, p)$-form $\omega$ on $M$, for some $1 \leq p \leq n-1$. Then
\[
\int_M (-dJ\theta)^{n-p} \wedge \omega > 0.
\]
On the other hand, by Stokes' theorem once again,
\[
\int_M (dJ\theta)^{n-p} \wedge \omega = \int_M (J\theta \wedge (dJ\theta)^{n-p-1}) \wedge d\omega = 0,
\]
which cannot be. Notice that, when $n\geq4$, this argument proves in particular that a compact Vaisman manifold cannot admit a holomorphic symplectic structure. Indeed, if $\Omega$ were holomorphic symplectic, then $n$ would be even and $\Omega^{\frac{n-2}{2}} \wedge \overline{\Omega}^{\frac{n-2}{2}}$ would be a positive closed $(n-2, n-1)$-form, falling thus in the case discussed already. In complex dimension 2, holomorphic symplectic structures and Vaisman metrics are compatible, since this is the case of the Kodaira surface. 
    
\item Assume there exists a strongly Gauduchon metric $\Omega$. Then, by definition, there exists an $(n, n-2)$-form $\alpha$ such that $\partial \Omega^{n-1}=\overline{\partial} \alpha$. Again by Stokes' theorem we have:
\begin{equation*}\label{sG}
\int_{M} \sqrt{-1}\overline{\partial}\theta^{0, 1} \wedge \alpha = \int_{M}  \sqrt{-1}\theta^{0, 1} \wedge \overline{\partial} \alpha = \int_{M}  \sqrt{-1}\theta^{0, 1} \wedge \partial \Omega^{n-1} =  \int_{M}  \sqrt{-1}\partial \theta^{0, 1} \wedge \Omega^{n-1}.
\end{equation*}
However, $\overline{\partial}\theta^{0, 1}=0$ and $-\sqrt{-1}\partial \theta^{0, 1}=-\frac{\mathrm{1}}{2} dJ\theta$ is a semi-positive $(1, 1)$-form. Then the first term vanishes, but the last term is strictly negative, a contradiction again. 
\end{enumerate}
\end{proof}

\begin{remark}
We can extend the argument above to any smooth modification $\pi\colon \hat{M} \rightarrow M$ of a Vaisman manifold. Let $E$ be the set of codimension $\geq 2$ such that $\hat{M} \setminus \pi^{-1}(E)$ and $M \setminus E$ are biholomorphic. Then $\tau:=-\pi^{*}(dJ\theta_0)$ is a semi-positive $(1, 1)$-form on $\hat{M} \setminus \pi^{-1}(E)$ and vanishes on the exceptional set $\pi^{-1}(E)$. The same arguments presented above are valid by integrating on $\hat{M}$ instead. 
\end{remark}

\begin{remark}
An alternative way of understanding the essence of the metric incompatibilities lies in the existence of special currents. More specifically, the form $\tau:=-dJ\theta_0$ (or $\tau:=-\pi^{*}(dJ\theta_0)$ in case of a modification $\pi\colon \hat{M} \rightarrow M$ of a Vaisman manifold) gives rise to a family of currents on $\hat{M}$ that will rule out the existence of the structures discussed. Indeed, for any $1 \leq p \leq n-1$, define the positive current of bidimension $(n-p, n-p)$ on $\hat{M}$ by $T_p(\eta)=\int_{\hat{M}}\eta \wedge \tau^p$. Since $\tau$ is exact, it is straightforward then that $T_p$ is the $(n-p, n-p)$-component of a boundary current $dS$. By \cite[Theorem 1.17]{aa}, this means that there are no $p$-K\" ahler forms on $\hat{M}$ for any $1 \leq p \leq n-1$; in particular, there are no balanced metrics. Similarly, since for any $2\leq p \leq n-1$, we have $\tau^p=\frac{(-1)^{p-1}}{p-1}dd^c\pi^*\omega^{p-1}$ thanks to \eqref{thm:vaisman-pp-eq1}, then $T_p$ is a positive current of bidimension $(n-p, n-p)$, component of a boundary current $dd^cS$, therefore there are no Hermitian metrics $\omega$ with $dd^c\omega^p=0$ for any $1\leq p \leq n-2$, see \cite[Theorem 2.1]{fgv}.
\end{remark}

\begin{remark}
The fact that a Vaisman metric $\omega_0$ satisfies $dd^c\omega_0\leq 0$ as in \eqref{thm:vaisman-pp-eq1}, namely is {\em plurinegative}, has strong consequences on compact complex threefolds: the Monge-Amp\`ere volumes of Hermitian metrics in the $dd^c$-class of $\omega_0$ are uniformly positive, see \cite[Proposition 3.10]{agl}.
\end{remark}

\begin{remark} Note that the incompatibility between astheno-K\" ahler metrics and Vaisman metrics was proven also in \cite[Theorem F]{chiose-rasdeaconu}.
\end{remark}

\begin{remark}
The statement of \ref{thm:vaisman-pp} works in a more general setting, namely for compact complex manifolds $X$ satisfying the following cohomological properties:
\begin{enumerate}
    \item the kernel of the natural map
    $$H_{BC}^{1, 1}(X) \rightarrow H^2_{dR}(X;\mathbb C)$$
    admits a non-zero class $[\alpha]_{BC}$ with a semi-positive representative $\alpha$;
    \item for $2 \leq p \leq n-1$, $[\alpha^p]_{BC}$ vanishes in $H^{p,p}_{BC}(X)$. 
\end{enumerate}
\end{remark}

\section{Locally conformally symplectic forms taming the complex structure and special Hermitian metrics}\label{sec:LCS-taming}

In the opposite direction with respect to the strengthening Vaisman condition, a possible way to weaken the LCK condition is to remove the compatibility between the locally conformally symplectic and the complex structures. In this section, we study locally conformally symplectic forms that just tame the complex structure, in the sense that the $(1,1)$-component is a positive form but there are possibly also $(2,0)$ and $(0,2)$-components. These structures were systematically studied in \cite{ad16, ad18, ad22} in complex dimension 2.
Note that the analogue of the tamed-to-compatible question \cite{li-zhang, st10} in the locally conformally setting has a negative answer, see \cite{ad16, bazzoni-marrero, angella-ugarte}: in particular, any compact complex non-Kähler surface admits a locally conformally symplectic structure taming the complex structure \cite{ad16}, but some Inoue surfaces, those of type $\mathcal{S}^{+}_{t}$ with $t \notin \mathbb{R}$, do not admit any locally conformally K\"ahler structure \cite{bel00}. Little is known, however, in higher dimensions.

\subsection{Obstructions to the existence of LCS forms taming complex structure}

We present below two situations in which LCS forms that tame the complex structure cannot exist on certain compact complex manifolds of arbitrary dimension, unless they are globally conformal to a Hermitian-symplectic form. The first result is a generalization of Vaisman's theorem \cite{vai80}, and the second one is a generalization of \cite[Lemma 3.1]{ovv} on locally trivial fibrations.  

We recall that a compact complex manifold is said to {\em satisfy the $dd^c$-lemma} if any $d^c$-closed, $d$-exact $(p,q)$-form is $dd^c$-exact too.

\begin{theorem}\label{generalized-Vaisman}
Let $(M, J)$ be a compact complex manifold satisfying the $dd^c$-lemma for $(1, 1)$-forms. Any locally conformally symplectic form taming the complex structure $J$ is globally conformally Hermitian-symplectic. 
\end{theorem}

\begin{proof}
Let $\Omega$ be the LCS form taming $J$ and $\theta=\theta^{(1,0)}+\theta^{(0,1)}$ its Lee form.
We claim that, under the assumption of the $dd^c$-lemma on $(1,1)$-forms, we can assume $dJ\theta=0$, up to global conformal changes. Indeed, since $\theta$ is closed, $dJ\theta=\sqrt{-1}\overline\partial\theta^{(1,0)}-\sqrt{-1}\partial\theta^{(0,1)}$ is an exact $(1, 1)$-form. In particular, it is clearly $d^c$-closed. Therefore, as the manifold satisfies the $dd^c$-lemma on $(1,1)$-forms, there exists a smooth function $f$ such that $dJ\theta=dJdf$.  We may change now $\Omega$ conformally by $e^{-f}\Omega$. This is still an LCS form taming $J$ with Lee form $\theta-df$ satisfying $dJ(\theta-df)=0$.

Just as in Vaisman's theorem, the trick henceforth is to calculate in two different ways the integral
\[
I := \int_M J\theta \wedge d \Omega^{n-1},
\]
where $n$ is the complex dimension of $M$.
On one hand, by Stokes' Theorem,
$$I = 0 .$$
On the other hand, 
\[
    I = \int_M J\theta \wedge d \Omega^{n-1} = (n-1) \int_M J\theta \wedge \theta \wedge \Omega^{n-1} = -2(n-1) \int_M \sqrt{-1} \theta^{(1, 0)} \wedge \overline{\theta^{(1, 0)}} \wedge \Omega^{n-1}.
\]
However, in the above only the $(n-1, n-1)$-component of $\Omega^{n-1}$ is relevant to the integration and this is 
\[
\left(\Omega^{n-1} \right)^{(n-1, n-1)} = \sum_{k=0}^{\lfloor \frac{n-1}{2} \rfloor} \binom{n-1}{k} (\Omega^{(2, 0)} \wedge \Omega^{(0, 2)})^k \wedge (\Omega^{(1, 1)})^{n - 2k - 1}.
\]
Therefore,
\begin{eqnarray*}
I
&=&- 2(n-1) \sum\limits_{k=0}^{\lfloor \frac{n-1}{2} \rfloor} \binom{n-1}{k} \int_M \sqrt{-1} \theta^{(1, 0)} \wedge \overline{\theta^{(1, 0)}} \wedge (\Omega^{(2, 0)} \wedge \Omega^{(0, 2)})^k \wedge (\Omega^{(1, 1)})^{n - 2k - 1} \\
&=&- 2(n-1) \sum_{k=0}^{\lfloor \frac{n-1}{2} \rfloor} \binom{n-1}{k} \int_M \sqrt{-1} \left( \theta^{(1, 0)} \wedge (\Omega^{(2,0)})^k \right) \wedge \overline{\left( \theta^{(1, 0)} \wedge (\Omega^{(2,0)})^k \right)} \wedge (\Omega^{(1, 1)})^{n - 2k - 1} ,
\end{eqnarray*}
where all the terms are negative because $\sqrt{-1} \left( \theta^{(1, 0)} \wedge (\Omega^{(2,0)})^k \right) \wedge \overline{\left( \theta^{(1, 0)} \wedge (\Omega^{(2,0)})^k \right)} \wedge (\Omega^{(1, 1)})^{n - 2k - 1}$ is positive for any $k \geq 0$.
Since $I=0$, all addends in the sum have to vanish, in particular the one corresponding to $k=0$ is $\|\theta^{(1,0)}\|^2_{\Omega^{(1, 1)}}=0$, then $\theta=0$. Hence $\Omega$ is Hermitian-symplectic. Therefore, any locally conformally symplectic form taming $J$ is therefore globally conformally Hermitian-symplectic. 
\end{proof}

\begin{remark}
In particular, \ref{generalized-Vaisman} implies that a non-K\" ahler Moishezon manifold does not admit strict LCS forms that tame the complex structure. Indeed, we recall that in \cite[Corollary 2.3]{pet} it is proven that a non-K\" ahler Moishezon manifold carries an exact, positive $(n-1, n-1)$-current, which means it cannot support Hermitian-symplectic structures either. 
\end{remark}

\begin{corollary}
A compact complex manifold $X$ admitting an LCS form taming the complex structure, which is not globally conformally Hermitian-symplectic satisfies the following cohomological conditions:
$$H^{1, 1}_{BC}(X) \neq 0,
\qquad
2h^{0, 1}_{\overline{\partial}}(X)>b_1(X) .$$
\end{corollary}

\begin{proof}
The first condition follows by noticing that $[dJ\theta] \in H^{1, 1}_{BC}(X)$ and that, by the proof of \ref{generalized-Vaisman}, $[dJ\theta]\neq 0$. The second inequality is a consequence of the $dd^c$-lemma on $(1, 1)$-forms being equivalent to $b_1(X)=2h_{\overline{\partial}}^{0,1}(X)$ (see \cite[I16]{gau84}). Indeed, denote by $h\colon H^{1, 1}_{BC}(X, \mathbb{R}) \rightarrow H^{1, 1}(X, \mathbb{R})$ the map induced by the identity. Then we have a short exact sequence given by
\begin{equation*}
    0 \rightarrow H^{1}(X, \mathbb{R}) \xrightarrow{i} H^{1}(X, \mathcal{O}) \xrightarrow{f} \mathrm{Ker}\, h \rightarrow 0,
\end{equation*}
where $i([\alpha])=[\alpha^{0, 1}]$ and $f([\alpha])=[\partial{\alpha}+\overline{\partial \alpha}]$. Since $0 \neq [dJ\theta] \in \mathrm{Ker}\, h$, we have then $2h^{0, 1}_{\overline{\partial}}(X)>b_1(X)$.
\end{proof}

\begin{proposition}\label{prop:LCS-taming-fibrations}
Let $\pi\colon M \to B$ be a locally trivial fibration between complex manifolds such that $B$ is path connected, $\Omega$ is an LCS form on $M$ taming the complex structure and the fibers are positive dimensional complex submanifolds of $M$. If $\pi^*\colon H^1(B) \to H^1(M)$ is an isomorphism, then $M$ is globally conformal to a Hermitian-symplectic metric.
\end{proposition}

\begin{proof}
Let $\theta$ be the Lee form of $\Omega$ and $p_M\colon \tilde{M} \to M$ the minimal covering of $M$ on which $p_M^* \theta$ is exact. Hence there exists $g\colon \tilde{M} \to \mathbb{R}$ smooth function such that $\Omega_0 := e^{-g} p_M^* \Omega$ is Hermitian-symplectic. 

As $\pi^*$ is an isomorphism, there exists a corresponding covering $p_B\colon \tilde{B} \to B$ and $\tilde{\pi}\colon \tilde{M} \to \tilde{B}$ such that the following diagram is commutative:
\begin{figure}[H]
		\centering
		\begin{tikzpicture}
		\matrix (m) [matrix of math nodes,row sep=3em,column sep=4em,minimum width=2em]
		{
			\tilde{M} & \tilde{B} \\
			M & B \\};
		\path[-stealth]
		(m-1-1) edge node [above] {$\tilde{\pi}$} (m-1-2)
		(m-1-1) edge node [left] {$p_M$} (m-2-1)
		(m-1-2) edge node [right] {$p_B$} (m-2-2)
		(m-2-1) edge node [above] {$\pi$} (m-2-2);
		\end{tikzpicture}
	\end{figure}
Denote by $\tilde{B}_0$ the set of regular values of $\tilde{\pi}$. As $\tilde{B}_0$ is connected, all fibers $F_b = \tilde{\pi}^{-1}(b)$ with $b \in \tilde{B}_0$ have the same complex dimension, which we denote by $m$, and represent the same homology class in $H_{2m}(\tilde{M})$, thus the value $\int_{F_b} \omega_0^m$ does not depend on $b \in \tilde{B}_0$.

Moreover, we claim that this value is strictly positive. Indeed, 
\begin{equation*}
    \int_{F_b}\Omega_0^m=\int_{F_b} (\Omega_0^m)^{(m, m)}=\sum_{k \geq 0}\int_{F_b} (\Omega_0^{(2,0)})^k \wedge \overline{(\Omega_0^{(0, 2)})^k} \wedge (\Omega_0^{(1, 1)})^{m-2k} \geq 0 ,
\end{equation*}
since $\Omega_0^{(1,1)}$ is positive.
Note that for $k=0$, we regain precisely the volume of $F_b$ with respect to the metric $\Omega_0^{(1, 1)}$, therefore $\int_{F_b} \Omega_0^{m}>0$. 

On the other hand, for any $\gamma \in \mathrm{Deck}(\tilde{M}/M)$, one can easily check that $\gamma$ permutes the fibers of $\tilde{\pi}$. Denote by $c\colon \mathrm{Deck}(\tilde{M}/M) \to \mathbb R^{>0}$ the character such that $\gamma^*\Omega_0=c(\gamma)\cdot\Omega_0$. Then we have
\[
\mathrm{Vol}_{\omega_0}(F_b)=\int_{F_b} \Omega_0^m = \int_{{F_{\tilde{\pi}^{-1}(\gamma(b))}}} \gamma^* (\Omega_0)^m = \int_{{F_{\tilde{\pi}^{-1}(\gamma(b))}}}{{c_\gamma^m}} \cdot \Omega_0^m = {{c_\gamma^m}} \int_{F_b} \Omega_0^m,
\]
so $c_\gamma = 1$ for all $\gamma \in \mathrm{Deck}(\tilde{M}/M)$ {\it i.e.} $\Omega$ is globally conformally Hermitian-symplectic.
\end{proof}

\subsection{Examples of LCS forms taming the complex structure}

The following results give ways of constructing new examples of manifolds admitting LCS forms taming $J$ out of old ones. One method is by blowing up special submanifolds and generalizes \cite[Theorem 2.8]{ovv}, and the other one is by analytic deformation.
Compare \cite[Proposition 3.1, Proposition 2.4]{yang} for the corresponding statement in the Hermitian-symplectic case.

\begin{proposition}\label{prop:LCS-taming-blowup}
Let $(M, J)$ be a complex manifold admitting a locally conformally symplectic structure $\Omega$ taming $J$ with Lee form $\theta$. The blow-up along a compact complex submanifold $N \subset M$ for which $\Omega_{|N}$ is globally conformally Hermitian-symplectic admits a locally conformally symplectic structure taming the induced complex structure.
\end{proposition}

\begin{proof}
Denote by $\pi: \tilde{M} \rightarrow M$ the blow-up of $M$ along $N$. Since $\Omega_{|N}$ is globally conformally Hermitian-symplectic, $\theta_{|N}$ is exact and choose $U$ a neighbourhood of $N$ such that $\theta_{|U}$ is still exact. Therefore, by rescaling $\Omega$, we can assume that $\theta_{|U}=0$ and hence, $\Omega_{|U}$ is a Hermitian-symplectic form. Take now $N\subset V \subset U$ and let $\pi: \tilde{U} \rightarrow U$ be the blow-up along $N$. Then by the proof of \cite[Proposition 3.2]{yang}, we get that there exists a Hermitian-symplectic form $\tilde{\Omega}$ on $\tilde{U}$ such that $\tilde{\Omega}_{|\tilde{U}\setminus \pi^{-1}(V)}=\pi^*(\Omega_{|U\setminus V})$. This means that $\tilde{\Omega}$ thus defined glues globally to $\pi^{*} \Omega$ on $\tilde{M}$ and satisfies $d\tilde{\Omega}=\pi^{*}\theta \wedge \tilde{\Omega}$, giving an LCS form that tames the induced complex structure $\tilde{J}$ on $\tilde{M}$. 
\end{proof}

\begin{proposition}\label{prop:LCS-taming-deformations}
Let $(M_t, J_t)_{t\in \mathbb{D}_\varepsilon}$ be an analytic family of compact complex manifolds such that $(M_0, J_0)$ admits a locally conformally symplectic form that tames $J_0$. Then for $\varepsilon>0$ small enough, $(M_t, J_t)$ can also be endowed with a locally conformally symplectic form taming $J_t$.
\end{proposition}

\begin{proof}
Let $\phi_t\colon M_t \rightarrow M_0$ be a diffeomorphism and $\Omega_0$ the LCS form that tames $J_0$ with Lee form $\theta_0$. Take $\Omega_t:=\phi_t^*\Omega_0$ and $\theta_t:=\phi_t^*\theta_0$. Then $d\Omega_t=\theta_t \wedge \Omega_t$ and $\theta_t$ is closed. Notice that $\phi_t$ is not a biholomorphism, hence $\Omega_t$ has a decomposition $\Omega_t=\Omega_t^{(2,0)}+\Omega_t^{(1,1)}+\Omega_t^{(0,2)}$ with respect to $J_t$, that is not necessarily the pullback via $\phi_t$ of the decomposition of $\Omega_0$ with respect to $J_0$. However, since $J_t \rightarrow J_0$ as $t \rightarrow 0$, we get that $\Omega_t(\phi^{-1}_{t,*}X, J_t\phi^{-1}_{t,*}X) \rightarrow \Omega_0(X, JX)>0$, as $t \rightarrow 0$. Therefore, for $\varepsilon>0$ small enough, $\Omega_t(X_t, J_tX_t)>0$, for any non-zero vector field $X_t$ on $M_t$ and thus, $\Omega_t^{(1, 1)}>0$ and $\Omega_t$ is an LCS form taming $J_t$.
\end{proof}

The following result is a partial converse to \ref{prop:LCS-taming-blowup} and it is analogue to \cite{miyaoka} for K\"ahler metrics, respectively to \cite[Theorem 4.7]{fino-tomassini} for pluriclosed metrics.

\begin{proposition}
Let $(M,J)$ be a complex manifold of complex dimension $n \geq 2$ and $\pi\colon \tilde M\to M$ be the blow-up of $M$ at a point $p \in M$. Then $\tilde M$ admits a Hermitian-symplectic structure if and only if $M$ admits a Hermitian-symplectic structure.
\end{proposition}

\begin{proof}
Let $\Omega$ be a Hermitian-symplectic structure on $\tilde M$, and consider its decomposition into pure-type components: $\Omega=\Omega^{(2,0)}+\Omega^{(1,1)}+\overline{\Omega^{(2,0)}}$. We know that $\partial\Omega^{(2,0)}=0$ and $\overline\partial\Omega^{(2,0)}=\partial\Omega^{(1,1)}$, in particular $\Omega^{(1,1)}$ is a pluriclosed metric on $\tilde M$.

If $E=\pi^{-1}(p)$ denotes the exceptional divisor of the blow-up, then $\Omega$ induces a Hermitian-symplectic structure on $M\setminus\{p\}\simeq\tilde M\setminus E$, and we want to extend it across $p$.
By \cite[Theorem 4.3]{fino-tomassini}, we know that $\Omega^{(1,1)}$ can be extended to a pluriclosed metric $\omega$ on $M$. Moreover, by the proof of \cite[Theorem 4.5]{fino-tomassini}, $\pi^*[\partial\omega] = [\partial\Omega^{(1,1)}] \in H^{2,1}_{\overline\partial}(\tilde M)$, where $\pi^*\colon H^{2,1}_{\overline\partial}(M) \to H^{2,1}_{\overline\partial}(\tilde M)$ is an isomorphism. Then $[\partial\omega]\in H^{2,1}_{\overline\partial}(M)$ vanishes: there exists a $(2,0)$-form $\alpha$ such that $\partial\omega=\overline\partial\alpha$. Therefore $\alpha+\omega+\overline\alpha$ is a Hermitian-symplectic form on $M$
\end{proof}

Apart from Inoue surface $\mathcal{S}_z$, with $z\notin \mathbb{R}$, \cite{inoue, bel00}, no other example of non-LCK manifold that supports taming LCS forms is known. Finding such examples  seems a rather challenging problem and we show below that the class of Kato manifolds in any complex dimension cannot help in this regard. Namely, we show that the existence of LCS taming structures on Kato manifolds is equivalent to the existence of LCK metrics.

Kato manifolds were introduced in \cite{kato} and are associated to a data consisting of a modification of the standard disc $\mathbb{B} \subseteq \mathbb{C}^n$ above the origin $\pi: \hat{\mathbb{B}} \rightarrow \mathbb{B}$ and a holomorphic embedding $\sigma: \mathbb{B} \rightarrow \hat{\mathbb{B}}$. The modification $\pi$ induces natural modifications $\pi: \hat{\mathbb{C}} \rightarrow \mathbb{C}$ and $\pi: \mathbb{C}\hat{\mathbb{P}}^n \rightarrow \mathbb{C}\mathbb{P}^n$, where $\mathbb{C}\hat{\mathbb{P}}^n$ is the compactification of $\hat{\mathbb{C}}$ with a hyperplane at infinity. A detailed description can be found in \cite{iopr}.

\begin{proposition}\label{prop:kato}
Let $X$ be the Kato manifold associated to the data $\pi: \hat{\mathbb{B}} \rightarrow \mathbb{B}$ and $\sigma: \mathbb{B} \rightarrow \hat{\mathbb{B}}$. Then $X$ admits an LCS form taming the complex structure if and only if $X$ admits a locally conformally K\" ahler metric. 
\end{proposition}

\begin{proof} We prove only the direct implication, since the converse is clear. Assume that $X$ admits an LCS form $\omega$ that tames the complex structure, with Lee form $\theta$. Let $\pi\colon \tilde{X} \rightarrow X$ be the universal cover of $X$. Then $e^{-f}\pi^*\omega$ is a globally conformal Hermitian-symplectic form on $\tilde{X}$. By \cite[Proposition 2.4]{iopr}, there exists an open holomorphic embedding of $\hat{\mathbb{B}}\setminus\{\sigma(0)\}$ into $\tilde{X}$, therefore $\hat{\mathbb{B}} \setminus \{\sigma(0)\}$ admits a Hermitian-symplectic metric, in particular $\hat{\mathbb{B}} \setminus \{\sigma(0)\}$ admits a pluriclosed metric. Consequently, by \cite[Theorem 4.3]{fino-tomassini}, $\hat{\mathbb{B}}$ admits a pluriclosed metric $\hat{\omega}$. Moreover, by \cite[Proposition 4.4]{fino-tomassini} we obtain that $\hat{\mathbb{C}}^n$ admits a pluriclosed metric $\tilde{\omega}$ that glues to the Fubini-Study metric on $\mathbb{C}\hat{\mathbb{P}}^n \setminus \hat{\mathbb{B}}$. Therefore, $\mathbb{C}\hat{\mathbb{P}}^n$ admits a pluriclosed metric. Finally, since $\mathbb{C}\hat{\mathbb{P}}^n$ is a Moishezon manifold, \cite[Theorem 2.2]{chiose} implies that $\mathbb{C}\hat{\mathbb{P}}^n$ is necessarily K\" ahler and \cite[Theorem 10.3]{iopr} concludes that $X$ is locally conformally K\" ahler.  
\end{proof}

\end{document}